\documentclass[12pt]{article}
\usepackage{amssymb,amsmath,latexsym}
\usepackage{amsmath,amsthm,amscd,amsfonts,amssymb,enumerate}
\usepackage{graphicx}
\usepackage{tikz}
\usepackage{float}
\usepackage{tkz-berge}
\usepackage{caption}
\usepackage[position=top]{subfig}

\usepackage{hyperref}
\usepackage{authblk}
\newtheorem{theorem}{Theorem}

\newtheorem{question}{Question}

\newtheorem{lemma}{Lemma}

\begin{document}
\title{Automorphism groups of $2$-groups of coclass at most $3$}
\author[1,2]{Alireza Abdollahi}
\affil[1]{Department of Mathematics, University of Isfahan, Isfahan, 81746-75441, Iran}
\affil[2]{School of Mathematics, Institute for Research in
	Fundamental Sciences (IPM), Tehran, Iran}
\author[1]{Nafiseh Rahmani}
\date{}
\maketitle
\abstract{It is proved that automorphism  groups of all $2$-groups of coclass $2$ are $2$-groups, except only three ones; for $2$-groups of coclass $3$, there are only $20$ groups and exactly six infinite sequences of $2$-groups whose automorphism groups are not $2$-groups.  }
\section{Introduction}
Many interesting questions in group theory are related to the order of automorphism groups. One of them which is considered in \cite{ma}   is the following.
\begin{question}\label{ques}
Is it true that for most $p$-groups the automorphism group is a $p$-group?
\end{question}
Where  by `most' it is meant to consider  all $p$-groups of order at most $p^n$ and let $n$ go to infinity.
The well-known examples of finite $2$-groups satisfying the property mentioned in  Question \ref{ques}  are  cyclic $2$-groups, dihedral $2$-groups $\mathbb{D}_{2^n}(n \geq 3)$ of order $2^n$ and the generalized quaternion groups $\mathbb{Q}_{2^n}(n \geq 4)$ of order $2^n$.
Also three infinite families of $2$-groups were introduced in \cite{cy}.\\
In \cite{H} and \cite{m}, it was proved that the automorphism group of a finite $p$-group is almost always a $p$-group,
 where interpretation of  `almost always'  depends on three parameters: the lower $p$-length, the number of generators, and $p$; while two of them are fixed, the third goes to infinity; however the latter interesting result does not answer  Question \ref{ques} and it still  remains open.
 
 The coclass of a group of order $p^n$ ($p$ prime) and  nilpotency class $c$ is defined to be $n-c$. 
 
 In this paper we try to determine all $2$-groups of coclass at most $3$ whose automorphism groups are $2$-groups.

Let us say some words on the classification of $p$-groups by coclass.  Leedham-Green
and Newman suggested to investigate $p$-groups by coclass \cite{ln}. The latter suggestion  began an important research
project and made a new insight into the theory of $p$-groups. See the book of Leedham-Green and McKay
\cite{lm} for more information.

An essential tool to study  $p$-groups of coclass $r$  is the directed graph $\mathcal{G}(p,r)$: its vertices correspond
to the isomorphism types of  finite $p$-groups of coclass $r$ and there is a directed edge from  $G$ to $H$ if $G\cong H/N$, where $N$ is a normal subgroup of order $p$ of $H$.
 Let $G$ and $H$ be two distinct vertices of $\mathcal{G}(p,r)$. We say that $H$ is a descendant of $G$ if
there exists a path from $G$ to $H$.

 Every infinite pro-$p$-group $S$ of coclass $r$ defines an infinite path in $\mathcal{G}(p,r)$ by taking the sequence of groups $S_u:=S/\gamma_u(S), S_{u+1}:=S/\gamma_{u+1}(S),\dots$, where $u$ is chosen large enough so that $S_u$ is a finite $p$-group of  coclass $r$ and $S_u$ is not isomorphic to a quotient of any other infinite pro-$p$-group $\overline{S}$ of coclass $r$ with $\overline{S}\ncong S$. Then the full subtree $\mathcal{T}(S)$ of $\mathcal{G}(p,r)$ consisting of all descendants of the root $S_u$
is called a coclass tree in $\mathcal{G}(p,r)$. The infinite path $S_u, S_{u+1}, S_{u+2},\dots$ contains every infinite path of $\mathcal{T}(S)$. The infinite path starting from $S_u$ is called the main line of $\mathcal{T}(S)$.  The pro-$p$-group $S$ can be reconstructed from the  coclass tree $\mathcal{T}(S)$ as the inverse
limit of the groups on the main line of $\mathcal{T}(S)$. Hence the main lines of $\mathcal{G}(p,r)$
correspond 1-1 to the isomorphism types of infinite pro-$p$-groups of coclass $r$. By Theorem $D$ of the coclass
theorems \cite{lm}, there are only finitely many isomorphism types of infinite pro-$p$-groups of coclass $r$.
Thus  $\mathcal{G}(p,r)$ consists of finitely many different  coclass trees.

 For $i\geq u$, the subgraph $\mathcal{B}_i$ of $\mathcal{T}(S)$, consisting of all descendants of $S_i$ which are not descendants
of $S_{i+1}$ is called a branch of  $\mathcal{T}(S)$.
We say that a vertex (a group) of $\mathcal{B}_i$ is capable if it has immediate descendants.\\
For $k\in\mathbb{N}$, we define shaved coclass tree $\mathcal{T}_k(S)$ as the full subgraph of $\mathcal{T}(S)$ consisting of those vertices (groups) which have distance at most $k$ from a vertex (a group) on the main line and we consider $\mathcal{B}_{k,i}$ as the $i$-th branch of $\mathcal{T}_k(S)$.
 
 In \cite{cl}, pro-$2$-groups of  coclass at most $3$ were determined and it was used to classify  $2$-groups of coclass at most $3$ into families (or coclass trees). Based upon these observations, the existence of preiodic patterns in graphs $\mathcal{G}(2,r)$ was conjectured.
   In  \cite{el}, it is proved  that $\mathcal{G}(2,r)$ could be described by  a finite number of  subgraphs and periodic patterns.

 The infinite pro-$2$-group $S$ corresponding to the tree $\mathcal{T}(S)$ is an extension of an infinite group $T$
(called translation subgroup) by a finite $2$-group $P$ (called point
group).  Furthermore, $T\cong \mathbb{Z}_2^d$, where $\mathbb{Z}_2$ is the group of  $2$-adic integers \cite[Theorem 7.4.12]{lm}. The dimension $d$ of $T$ is also called the rank of $\mathcal{T}(S)$. The following is the main result of \cite{el} that was summarised  in \cite[Theorem 1]{E}.
\begin{theorem}\label{1}
Let $\mathcal{T}(S)$ be a coclass tree in $\mathcal{G}(2,r)$ of rank $d$. Then $\mathcal{T}(S)$ is virtually periodic with periodicity $d$; that is, there exists $f \in \mathbb{N}$ such that for every $i \geq f$ there is a graph isomorphism
$$\pi_i:\mathcal{B}_i\rightarrow \mathcal{B}_{i+d}: G \mapsto G^{\pi_i}$$
and $|G^{\pi_i}| = |G|2^d$ for all $G \in \mathcal{B}_i$.
\end{theorem}
 By Theorem \ref{1}, also for each integer $k$, the subtree $\mathcal{T}_k(S)$ is virtually periodic. It allows to define infinite sequences of  finite $2$-groups of coclass $r$ in $\mathcal{T}_k(S)$. For suitable $m \geq f$,  every group (vertex)  $G$ of $\mathcal{B}_{k,i}$ with $m \leq i \leq m + d$, defines an associated infinite  sequence $( G_0, G_1, G_2, \dots)$  which contains $G$ and all iterated preimages of $G$ under the maps $\pi_{i+j}$ for $j \in\mathbb{N}$. So the subtree $\mathcal{T}_k(S)$  divides into finitely many infinite  sequences and finitely many groups lying outside these sequences.
It was proved in \cite{el} that the groups in an infinite  sequences could be described by a single parametrised presentation and it could be constructed by a finite calculation. In \cite{bd}   a detailed description for constructions of the infinite  sequences was offered  and (as  an example)   the presentations of the infinite  sequences for the graphs  $\mathcal{G}(2,1)$, $\mathcal{G}(3,1)$ and $\mathcal{G}(2,2)$ were determined. Note that the presentations of infinite  sequences of $2$-groups of coclass $3$ have not yet known.

A significant application  of coclass theory  has been achieved by Eick \cite{E}.
 She considered infinite  sequences that were defined according to periodic property of  coclass trees to show the following. For a group $G$, $Aut(G)$ denotes the automorphism group of $G$.
\begin{theorem}[{\cite[Theorem 2]{E}}]\label{e}
Let  $(G_j  \mid j\in \mathbb{N})$ be an infinite  sequence of $2$-groups of coclass $r$. Then there exist natural numbers $m$ and $f$ such that $|Aut(G_{j+1})| = 2^m|Aut(G_j)|$ for every $j \geq f$.
\end{theorem}

To determine all $2$-groups of a fixed coclass  having $2$-power order  automorphism groups, in   Eick's approach,  one needs to consider  the orders of automorphism groups of all groups in periodic parts of  coclass trees. In this paper we use an observation (Lemma \ref{l}) to find these groups in a  straightforward way. It follows  that if the automorphism group of a $2$-group in the main line of a coclass tree is  $2$-group, then the number of $2$-groups of the coclass tree whose automorphism groups are not $2$-groups is finite.

Our main results are the following.

\begin{theorem}\label{c2}
	The only $2$-groups of coclass $2$ whose automorphism groups are not $2$-groups are as follows.
	\begin{enumerate}
		\item
		$C_2\times C_2\times C_2$,
		\item
		$\mathbb{Q}_8\times C_2$
		\item
		 $\langle a,b,c \mid a^4=b^2=c^2=1, b^a=b, c^a=c, b^c=a^2b\rangle$.
	\end{enumerate}
\end{theorem}

\begin{theorem}\label{5}
	In the graph $\mathcal{G}(2,3)$, among  all coclass trees, except one coclass tree (Family 81 of \cite{cl}),  there are finitely many  groups whose  automorphism groups  are not $2$-groups.
\end{theorem}

\begin{theorem}\label{s}
	There are exactly six infinite sequences of $2$-groups of coclass $3$ (in  Family $81$ of \cite{cl}) whose automorphism groups are not $2$-groups.
\end{theorem}
 \section{ Proof of the main results}
In order to prove our results,  we need the following  lemma.
\begin{lemma}\label{l}
 Suppose that  $G$ and $H$ are two vertices  in the coclass graph $\mathcal{G}(2,r)$ such that there is a directed edge from $G$ to $H$. If $Aut(G)$ is a $2$-group, then so is $Aut(H)$.
\end{lemma}
\begin{proof}
By the definition of graph $\mathcal{G}(2,r)$,   $G\cong H/N$, where  $N$  is a normal subgroup of order $2$ of $H$. Since $G$ and $H$ have the same coclass $r$,   $N$ is the last non-trivial term of the lower central series of $H$. Thus  $N^\alpha=N$ for all automorphisms $\alpha$ of $H$ and so $\alpha$ acts trivially on $N$ as $|N|=2$. By hypothesis the order of the induced automorphism $ \bar{\alpha}$ on $H/N$ by $\alpha$  is $2^k$ for some non-negative integer $k$. Now for every $x\in H$ there exists some $z\in N$  such that $x^{\alpha^{2^k}}= xz$. Hence  $x^{\alpha^{2^{k+1}}}= x$, since $z^\alpha=z$. So the order of $\alpha$ is a power of $2$. This completes the proof.
\end{proof}

In fact, by  Lemma \ref{l}, if there exists a group in a main line  whose automorphism group is $2$-group, then all its descendants have the same property. In \cite{cl} all infinite pro-$2$-groups of coclass at most 3
have been determined. Since every infinite pro-$2$-group corresponds to one main line, we have  the number of main lines.
\begin{enumerate}
\item
There is 1 main line in $\mathcal{G}(2,1)$;
\item
There are 5 main lines in $\mathcal{G}(2,2)$;
\item
There are 54 main lines in $\mathcal{G}(2,3)$.
\end{enumerate}
The $2$-groups of coclass $1$ were classified \cite{B} and the graph $\mathcal{G}(2,1)$ was known. The dihedral group  $\mathbb{D}_8$ lies in the main line  of $\mathcal{G}(2,1)$ and $Aut(\mathbb{D}_8)\cong \mathbb{D}_8$. Hence  Lemma \ref{l} simply implies that   the automorphism groups of all its descendants (dihedral $2$-group $\mathbb{D}_{2^n}(n \geq 4)$, the generalized quaternion group $\mathbb{Q}_{2^n} (n\geq 4)$ and semi dihedral group $\mathbb{SD}_{2^n}(n \geq 4)$) are $2$-groups.

In the following, we determine all $2$-groups of coclass $2$ whose automorphism groups are $2$-groups.  Note that by $\sf{SmallGroup}(n, m)$, we mean the $m$th group of order $n$ in the small groups library of GAP \cite{G}.
\begin{proof}[Proof of Theorem \ref{c2}]
By using GAP \cite{G}, we find out that the automorphism groups of all $2$-groups of coclass $2$ and order $2^5$ are $2$-groups. It follows from the classification of $2$-groups of coclass $2$ \cite{cl} that every main line in the graph $\mathcal{G}(2,2)$ has a group of order $2^5$. Thus  Lemma \ref{l} implies  that
 all $2$-groups  of coclass $2$ and  order at least $2^5$ have  $2$-power order automorphism groups. Now by GAP \cite{G} it is easy to check that    $2$-groups   $\sf{SmallGroup}(8, 5), \sf{SmallGroup}(16, 12)$ and $\sf{SmallGroup}(16, 13)$ are the only $2$-groups of coclass $2$ whose automorphism groups are not $2$-groups. This completes the proof.
\end{proof}
        In \cite{cl}, $2$-groups of coclass $3$ were classified into $54$ families (or coclass trees). Also a conjectured descendant pattern was presented for every family.  In Theorems \ref{5} and \ref{s},  we find  that one of the coclass tree  contains six infinite sequences of $2$-groups whose automorphism groups are not $2$-groups. In the other coclass trees, we show that the number of $2$-groups of coclass $3$ whose automorphism groups are not $2$-groups is finite.

\begin{proof}[Proof of Theorem \ref{5}]
By  GAP \cite{G}, we have checked the automorphism group orders of $2$-groups of coclass $3$. Among $2$-groups of order $2^8$, we find that there are six groups whose automorphism groups are not $2$-groups. By an application of ANUPQ Package \cite{an} we see that they  are immediate descendants of the group  $\sf{SmallGroup(2^7, 2140)}$. One of them is $\sf{SmallGroup(2^8, 26963)}$ which is capable and having $10$ immediate descendants. The five  others have no immediate descendant. The pattern (see Figure \ref{fig1}), one capable group with $10$ immediate descendants, is repeated  and it is similar to  the descendant patterns  of  Family $55$ and  Family $81$ of \cite{cl}.
 By using the ANUPQ Package \cite{an}, we computed the quotients of order $2^7$ and  $2^8$ of  two latter families and we see that  $\sf{SmallGroup(2^7, 2140)}$ and  $\sf{SmallGroup(2^8, 26963)}$ are  $2$-quotients of  family $81$.
 
According to \cite[Theorem 4.4]{cl},
we must consider  $2$-groups of order at least $2^{15}$ to find at least one group in every main line whose automorphism group is $2$-group. We use GAP \cite{G} and  ANUPQ Package \cite{an} to access these groups.
 By these tools we see that there are only six $2$-groups of  order $2^{15}$ whose automorphism groups are not $2$-groups and these are descendants of capable group $\sf{SmallGroup(2^8, 26963)}$.
 Thus we find that except  one coclass tree (Family $81$ in \cite{cl}), in every main line of the other $53$ coclass trees, there is a $2$-group of order $2^{15}$ whose automorphism group is a $2$-group. So it follows from Lemma \ref{l} that among all $53$ coclass trees there are finitely many  groups  whose automorphism groups  are not $2$-groups.
\end{proof}
\begin{figure}[H]
	\centering
	\captionsetup{justification=centering}
	\includegraphics[width=0.3\textwidth]{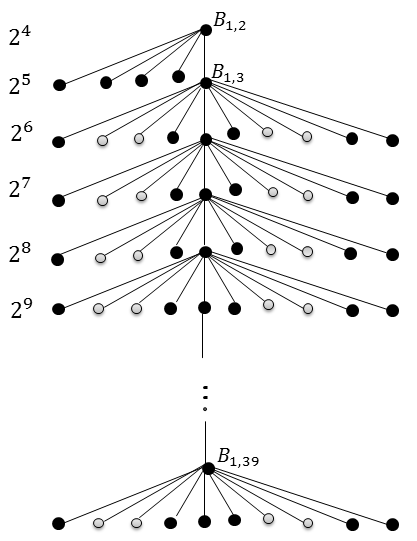}
	\caption{\label{fig1}\\$\circ :$ The groups whose automorphism groups are $2$-groups \\
		\quad\quad$\bullet:$ The groups whose automorphism groups are not $2$-groups}
\end{figure}
\begin{proof}[Proof of Theorem \ref{s}]
 It follow from Theorem \ref{5} that  there are six $2$-groups of  order $2^{15}$ whose automorphism groups are not $2$-groups and these groups belong to  Family $81$ in \cite{cl}.\\
 Let $S$ be the infinite pro-2-group of coclass 3 corresponding to  Family 81. The pro-$2$-group $S$ has dimension $d=1$ and it has the following pro-$2$-presentation (See \cite[Table 2 and Appendix A]{cl}),
 \begin{align*}
\{ t,a,u,v \quad\mid\quad &a^2=1,\quad  u^2=1, \quad v^2=1,\\
&u^{-1}vu=v,\quad a^{-1}ua=u,\quad a^{-1}va=v, \\
 &a^{-1}ta=t^{-1},\quad t^{-1}ut=u, \quad t^{-1}vt=v\}.
\end{align*}
 We first investigate descendant pattern of  $\mathcal{T}_k(S)$ for $k=1$ to find its periodic part. Then we use the periodicity to determine presentations of the infinite sequences of groups in $\mathcal{T}_1(S)$.\\
Since $\gamma_2(S)$ is torsion-free abelian and $S_2$ is a finite $2$-group of coclass $3$, so we have the  primary root, $u=2$.
By Theorem \ref{e}, if we know about the parameter $f$ then we can  determine the periodic part of $\mathcal{T}_1(S)$. In \cite[section 7]{el},  a possible upper bound of $f$ was introduced by means of some parameters $l, a, b, e, d$. We use the inequality  $l\geq u-1+ max\{p^r, 3/2(k+d)\}$ to get $l=9$. Let $R:=S_{9}$ and $T:=\gamma_{9}(S)$. By computing other parameters we find $2$-groups in $\mathcal{B}_{1,39}$ laying in the periodic part. Using ANUPQ Package one can see a descendant pattern in the periodic  part as in Figure \ref{fig1}. It is similar to the conjectured descendant pattern \cite{cl}( one capable group with $10$ immediate descendants).  As   this pattern begins from $2$-groups of order $2^5$,  the subtree $\mathcal{T}_1(S)$  has periodic root $2^5$.\\
  Now, we determine  presentations of the six infinite sequences.
According to \cite{bd}, we should determine some parameters: primary root($u$), secondary root($l$) and offset($o$).\\
As we have computed,   the primary root  $u$ is  $2$ and the secondary root $l=9$. By the structure of $\mathcal{T}_1(S)$,  we can take $l=3$ as the  secondary root.
 So we assume $R=S_3$ and $T=\gamma_3(S)$. Then the kernel $H$ of the action of $R$ on $T$ is $H\cong  C_4\times C_2\times C_2$. By   \cite[Definition 6 and Theorem 7]{bd}, we can choose $o=8$ as the offset.
    It follows that $A_i:=T/T_{o+id}=T/T_{8+i}$ $(i\geq 0)$ is cyclic of order $2^{8+i}$, where $T_j:=\gamma_{l+j}(S)$.
Since the groups in an infinite sequence $\mathcal{F}:=( G_t \mid t\in \mathbb{N}_0)$ are defined as extension of $A_i$ by $R$,  we need to compute cocycle $\delta_i\in H^2(R,A_i)$.
Using the tools in \cite{Ha}, we find that the every extension which is defined by
an element  in $Z^2(R,A_i)$  has a  presentation as follows on the generators $g_1, g_2, g_3, g_4, g_5,t$ with relations
\begin{align*}
  &g_1^2=t^{x_1},  \\
&g_2^{g_1}=g_2g_3t^{x_2},\quad g_2^2=g_3t^{x_3},\\
 &g_3^{g_1}=g_3t^{x_4}, \quad\quad g_3^{g_2}=g_3t^{x_5}, \quad \quad g_3^2=t^{x_6}, \\
 &g_4^{g_1}=g_4t^{x_7},\quad \quad g_4^{g_2}=g_4t^{x_8}, \quad\quad g_4^{g_3}=g_4t^{x_9}, \quad\quad g_4^2=1t^{x_{10}}, \\
 &g_5^{g_1}=g_5t^{x_{11}}, \quad\quad g_5^{g_2}=g_5t^{x_{12}}, \quad g_5^{g_3}=g_5t^{x_{13}}, \quad\quad g_5^{g_4}=g_5t^{x_{14}}, \quad g_5^2=t^{x_{15}},  \\
 &t^{g_1}=t^{-1}, \quad\quad\quad t^{g_2}=t, \quad\quad\quad t^{g_3}=t, \quad\quad\quad\quad t^{g_4}=t, \quad\quad\quad\quad t^{g_5}=t, \quad\quad t^{2^{8+i}}=1
\end{align*}
where $x_1,x_2,\dots,x_{15}\in \{0,\dots,2^{8+i}-1\}$. So every element in $Z^2(R,A_i)$ corresponds to a vector $(x_1,x_2,\dots,x_{15})$ which defines a consistent presentation. A straightforward computation to test the consistency  yields $10$ infinite sequences. Now, Theorem \ref{e} shows that the automorphism groups  of which  sequences  are not $2$-groups. Therefore, the followings are six infinite sequences of $2$-groups of coclass $3$  whose automorphism groups are not $2$-groups.
 \begin{align*}
K_{n}^1=\langle g_1,g_2,g_3,g_4,g_5,t \quad\mid\quad &g_1^2=1, \\
&g_2^{g_1}=g_2g_3,\quad g_2^2=g_3t,\\
 &g_3^{g_1}=g_3t, \quad g_3^{g_2}=g_3, \quad g_3^2=t^{-1}, \\
 &g_4^{g_1}=g_4, \quad g_4^{g_2}=g_4, \quad g_4^{g_3}=g_4, \quad g_4^2=1, \\
 &g_5^{g_1}=g_5, \quad g_5^{g_2}=g_5, \quad g_5^{g_3}=g_5, \quad g_5^{g_4}=g_5, \quad g_5^2=1,  \\
 &t^{g_1}=t^{-1}, \quad t^{g_2}=t, \quad\quad t^{g_3}=t, \quad\quad t^{g_4}=t, \quad t^{g_5}=t, \quad t^{2^{n+1}}=1\rangle
\end{align*}
\begin{align*}
K_{n}^2=\langle g_1,g_2,g_3,g_4,g_5,t \quad\mid\quad &g_1^2=1,\\
&g_2^{g_1}=g_2g_3, \quad g_2^2=g_3t^{1+2^n},\\
 &g_3^{g_1}=g_3t, \quad g_3^{g_2}=g_3, \quad g_3^2=t^{-1},\\
 &g_4^{g_1}=g_4, \quad g_4^{g_2}=g_4, \quad g_4^{g_3}=g_4, \quad g_4^2=1,\\
 &g_5^{g_1}=g_5, \quad g_5^{g_2}=g_5, \quad g_5^{g_3}=g_5, \quad g_5^{g_4}=g_5, \quad g_5^2=1, \\
 &t^{g_1}=t^{-1}, \quad t^{g_2}=t, \quad\quad t^{g_3}=t, \quad\quad t^{g_4}=t, \quad t^{g_5}=t, \quad t^{2^{n+1}}=1\rangle
\end{align*}
\begin{align*}
K_{n}^3=\langle g_1,g_2,g_3,g_4,g_5,t \quad\mid\quad &g_1^2=t^{2^n},\\
&g_2^{g_1}=g_2g_3, \quad g_2^2=g_3t,\\
&g_3^{g_1}=g_3t, \quad g_3^{g_2}=g_3, \quad g_3^2=t^{-1},\\
 &g_4^{g_1}=g_4, \quad g_4^{g_2}=g_4, \quad g_4^{g_3}=g_4, \quad g_4^2=1,\\
 &g_5^{g_1}=g_5, \quad g_5^{g_2}=g_5,\quad g_5^{g_3}=g_5,\quad g_5^{g_4}=g_5, \quad g_5^2=1, \\
  &t^{g_1}=t^{-1}, \quad t^{g_2}=t,\quad t^{g_3}=t,\quad\quad t^{g_4}=t,\quad\quad t^{g_5}=t, \quad t^{2^{n+1}}=1\rangle
\end{align*}
\begin{align*}
K_{n}^4=\langle g_1,g_2,g_3,g_4,g_5,t \quad\mid\quad &g_1^2=1,\\
 &g_2^{g_1}=g_2g_3, \quad g_2^2=g_3t,\\
 &g_3^{g_1}=g_3t, \quad g_3^{g_2}=g_3, \quad g_3^2=t^{-1},\\
 &g_4^{g_1}=g_4, \quad g_4^{g_2}=g_4, \quad g_4^{g_3}=g_4, \quad g_4^2=1,\\
 &g_5^{g_1}=g_5, \quad g_5^{g_2}=g_5,\quad g_5^{g_3}=g_5, \quad g_5^{g_4}=g_5t^{2^n},\quad g_5^2=1, \\
 &t^{g_1}=t^{-1}, \quad t^{g_2}=t, \quad t^{g_3}=t, \quad t^{g_4}=t, \quad t^{g_5}=t, \quad t^{2^{n+1}}=1\rangle
\end{align*}
\begin{align*}
K_{n}^5=\langle g_1,g_2,g_3,g_4,g_5,t \quad\mid\quad &g_1^2=1,\\
&g_2^{g_1}=g_2g_3, \quad g_2^2=g_3t^{1+2^n},\\
&g_3^{g_1}=g_3t, \quad g_3^{g_2}=g_3,  \quad g_3^2=t^{-1},\\
&g_4^{g_1}=g_4,\quad g_4^{g_2}=g_4, \quad g_4^{g_3}=g_4, \quad g_4^2=1,\\
 &g_5^{g_1}=g_5, \quad g_5^{g_2}=g_5,\quad g_5^{g_3}=g_5,\quad g_5^{g_4}=g_5t^{2^n},\quad g_5^2=1, \\
  &t^{g_1}=t^{-1}, \quad t^{g_2}=t, \quad t^{g_3}=t,\quad t^{g_4}=t, \quad t^{g_5}=t, \quad t^{2^{n+1}}=1\rangle
\end{align*}
\begin{align*}
K_{n}^6=\langle g_1,g_2,g_3,g_4,g_5,t \quad\mid\quad &g_1^2=t^{2^n},\\
&g_2^{g_1}=g_2g_3, \quad g_2^2=g_3t,\\
&g_3^{g_1}=g_3t, \quad g_3^{g_2}=g_3, \quad g_3^2=t^{-1},\\
&g_4^{g_1}=g_4, \quad g_4^{g_2}=g_4, \quad g_4^{g_3}=g_4, \quad g_4^2=1,\\
 &g_5^{g_1}=g_5,\quad g_5^{g_2}=g_5, \quad g_5^{g_3}=g_5, \quad g_5^{g_4}=g_5t^{2^n}, \quad g_5^2=1, \\ &t^{g_1}=t^{-1}, \quad t^{g_2}=t, \quad t^{g_3}=t, \quad t^{g_4}=t, \quad t^{g_5}=t, \quad t^{2^{n+1}}=1\rangle
 \end{align*}

	Consider $\mathcal{T}_2(S)$.  As the proof of Theorem \ref{s}, we find that  the parameter $l=9$, when $k=2$.  So the other parameters $a, b, e,d$ are the same as those  of $\mathcal{T}_1(S)$ and $39$-th branch of $\mathcal{T}_2(S)$ lies in  the periodic part. By using ANUPQ Package, we find that there is no group in $\mathcal{B}_{2,39}$ and  $\mathcal{T}_2(S)$  with distance $2$  from the main line.   Therefore in the coclass tree $\mathcal{T}(S)$, the maximum distance of a group from a mainline group is $1$ and $\mathcal{T}(S)=\mathcal{T}_1(S)$. So according to Theorem \ref{5} and Theorem \ref{s} we have the only $2$-groups of coclass $3$  whose automorphism groups are not $2$-groups are $\sf{SmallGroup}(16, i)$, $\sf{SmallGroup}(32, j)$, $\sf{SmallGroup}(64, k)$ and $\sf{SmallGroup}(128, l)$ for $\sf{i}\in\{2, 10, 14\}$, $\sf{j}\in\{2, 26, 27, 33, 34, 38, 46, 47, 48, 49, 50\}$, $\sf{k}\in\{138$, $139, 251, $$ 258\}$ and $\sf{l}\in\{934,937\}$, and the six infinite sequences of $2$-groups presented in the proof of Theorem \ref{s}.
\end{proof}
\section*{Acknowledgements}
   The first author was supported in
part by Grant No. 96050219 from School of Mathematics, Institute for Research in Fundamental Sciences (IPM). The first
author was additionally financially supported by the Center of Excellence for Mathematics at the University of Isfahan.

\end{document}